\documentclass[11pt]{article}

\usepackage{amsmath,amssymb,amsfonts,amsthm}
\usepackage{amssymb,color}
\usepackage{graphics,graphicx}                 
\usepackage{xcolor}
\usepackage[usenames,dvipsnames]{pstricks}
\usepackage{epsfig}
\usepackage{pst-grad}
\usepackage{pst-plot}
\usepackage{multicol}
\usepackage{mathdots}
\usepackage{pst-node}
\usepackage{graphicx}
\usepackage{xcolor}
\usepackage{pgf}

\usepackage{tikz}
\usetikzlibrary{automata}
\usetikzlibrary{arrows}
\usetikzlibrary{positioning,calc}
\usetikzlibrary{graphs}
\usetikzlibrary{graphs.standard}
\usetikzlibrary{arrows,decorations.markings}
\usepackage{tkz-graph}
\usetikzlibrary{chains,fit,shapes}
\usetikzlibrary{calc}

\tikzset{every loop/.style={min distance=10mm,looseness=10}}
\tikzset{every state/.style={minimum size=2mm}}

\newcount\nodecount
\tikzgraphsset{
  declare={subgraph N}%
  {
    [/utils/exec={\global\nodecount=0}]
    \foreach \nodetext in \tikzgraphV
    {  [/utils/exec={\global\advance\nodecount by1}, 
      parse/.expand once={\the\nodecount/\nodetext}] }
  },
  declare={subgraph C}%
  {
    [cycle, /utils/exec={\global\nodecount=0}]
    \foreach \nodetext in \tikzgraphV
    {  [/utils/exec={\global\advance\nodecount by1}, 
      parse/.expand once={\the\nodecount/\nodetext}] }
  },
   declare={subgraph W}%
  {
    [/utils/exec={\global\nodecount=0}]
    \foreach \nodetext in \tikzgraphV
    {  [/utils/exec={\global\advance\nodecount by1}, 
      parse/.expand once={\the\nodecount/\nodetext}] }
  }
}

\newtheorem{theorem}{Theorem}
\newtheorem{lemma}[theorem]{Lemma}
\newtheorem{problem}{Problem}
\newtheorem{remark}[theorem]{Remark}

\title{On semi-transitive orientability of triangle-free graphs}

\author{Sergey Kitaev\footnote{Department of Mathematics and Statistics, University of Strathclyde, 26 Richmond Street, Glasgow G1, 1XH, United Kingdom. 
{\bf Email:} sergey.kitaev@strath.ac.uk.}\ \ and Artem Pyatkin\footnote{Sobolev Institute of Mathematics, Koptyug ave, 4, Novosibirsk, 630090, Russia}\ \footnote{Novosibirsk State University, Pirogova str. 2, Novosibirsk, 630090, Russia. {\bf Email:} artem@math.nsc.ru.}}

\begin{document}

\maketitle

\begin{abstract}
An orientation of a graph is {\it semi-transitive} if it is acyclic, and for any directed path $v_0\rightarrow v_1\rightarrow \cdots\rightarrow v_k$ either there is no arc between $v_0$ and $v_k$, or $v_i\rightarrow v_j$ is an arc for all $0\leq i<j\leq k$. An undirected graph is semi-transitive if it admits a semi-transitive orientation. Semi-transitive graphs generalize several important classes of graphs and they are precisely the class of word-representable graphs studied extensively in the literature.  

Determining if a triangle-free graph is semi-transitive is an NP-hard problem. The existence of non-semi-transitive triangle-free graphs was established via Erd\H{o}s' theorem by Halld\'{o}rsson and the authors in 2011. However, no explicit examples of such graphs were known until recent work of the first author and Saito who have shown computationally that a certain subgraph on 16 vertices of the triangle-free Kneser graph $K(8,3)$ is not semi-transitive, and have raised the question on the existence of smaller triangle-free non-semi-transitive graphs. In this paper we prove that the smallest triangle-free 4-chromatic graph on 11 vertices (the Gr\"otzsch graph) and the smallest triangle-free 4-chromatic 4-regular graph on 12 vertices (the Chv\'atal graph) are not semi-transitive. Hence, the Gr\"otzsch graph is the smallest triangle-free non-semi-transitive graph. We also prove the existence of  semi-transitive graphs of girth 4 with chromatic number 4 including a small one (the circulant graph $C(13;1,5)$ on 13 vertices) and dense ones (Toft's graphs). Finally, we show that each $4$-regular circulant graph (possibly containing triangles) is semi-transitive.\\


\end{abstract}

\section{Introduction}
An orientation of a graph is {\em semi-transitive} if it is acyclic, and for any directed path $v_0\rightarrow v_1\rightarrow \cdots \rightarrow v_k$ either there is no arc between $v_0$ and $v_k$, or $v_i\rightarrow v_j$ is an arc for all $0\leq i<j\leq k$. An undirected graph is {\em semi-transitive} if it admits a semi-transitive orientation. The notion of a semi-transitive orientation generalizes that of a {\em transitive orientation}; it was introduced by Halld\'{o}rsson, Kitaev and Pyatkin \cite{HKP11} in 2011 as a powerful tool to study {\em word-representable graphs} defined via alternation of letters in words and studied extensively in  recent years (see \cite{K17,KL15}).  The hereditary class of semi-transitive graphs is precisely the  class of word-representable graphs, and its significance is in the fact that it generalizes several important classes of graphs. In particular, we have the following useful fact.

\begin{theorem}[\cite{HKP16}]\label{3-color-graphs}  Any $3$-colourable graph is semi-transitive. \end{theorem}

A {\em shortcut} $C$ in a directed acyclic graph is an induced subgraph on vertices  $\{v_0,v_1,\ldots, v_k\}$ for $k\geq 3$ such that $v_0\rightarrow v_1\rightarrow \cdots \rightarrow v_k$ is a directed path, $v_0\rightarrow v_k$ is an arc, and there exist $0\leq i<j\leq k$ such that there is no arc between $v_i$ and $v_j$. The arc $v_0\rightarrow v_k$ in $C$ is called the {\em shortcutting arc}, and the path $v_0\rightarrow v_1\rightarrow \cdots \rightarrow v_k$ is the {\em long path} in $C$. Thus, an orientation is semi-transitive if and only if it is acyclic and shortcut-free. 

The following lemma is an easy, but very helpful observation that will be used many times in this paper. Note that it was first proved in \cite{AKM15} for the case of $m=4$.


\begin{lemma}[\cite{AKM15}]\label{lemma} Suppose that an undirected graph $G$ has a cycle $C=x_1x_2\cdots x_mx_1$, where $m\geq 4$ and the vertices in $\{x_1,x_2,\ldots,x_m\}$ do not induce a clique in $G$.  If $G$ is oriented semi-transitively, and $m-2$ edges of $C$ are oriented in the same direction (i.e. from $x_i$ to $x_{i+1}$ or vice versa, where the index $m+1:=1$) then the remaining two edges of $C$ are oriented in the opposite direction.\end{lemma}

\begin{proof} Clearly, if all arcs of $C$ have the same direction then we obtain a cycle; if $m-1$ arcs of $C$ have the same direction, we obtain a shortcut. So, the direction of both remaining arcs must be opposite.
\end{proof}

\begin{figure}
\begin{center}
\begin{tikzpicture}[node distance=1cm,auto,main 
node/.style={circle,draw,inner sep=1pt,minimum size=2pt}]

\node[main node] (1) {{\tiny 1}};
\node[main node] (2) [right of=1,xshift=5cm] {{\tiny 2}};

\node[main node] (3) [below of=1,yshift=0.3cm] {{\tiny  3}};
\node[main node] (4) [below of=3,yshift=0.3cm] {{\tiny 4}};
\node[main node] (5) [below of=4,yshift=0.3cm] {{\tiny 5}};
\node[main node] (6) [below of=5,yshift=0.3cm] {{\tiny 6}};

\node[main node] (7) [below of=2,yshift=0.3cm] {{\tiny 7}};
\node[main node] (8) [below of=7,yshift=0.3cm] {{\tiny  8}};
\node[main node] (9) [below of=8,yshift=0.3cm] {{\tiny 9}};
\node[main node] (10) [below of=9,yshift=0.3cm] {{\tiny 10}};

\node[main node] (11) [below of=6,yshift=0.3cm] {{\tiny 11}};

\node[main node] (13) [below of=11,yshift=0.3cm] {{\tiny 13}};
\node[main node] (14) [below of=13,yshift=0.3cm] {{\tiny 14}};

\node[main node] (12) [below of=10,yshift=0.3cm] {{\tiny 12}};

\node[main node] (15) [below of=12,yshift=0.3cm] {{\tiny 15}};
\node[main node] (16) [below of=15,yshift=0.3cm] {{\tiny 16}};

\path
(1) edge (2)
      edge (3)
      edge  [bend right=60] (4)
      edge [bend right=60] (5)
      edge [bend right=60] (6);
\path
(2) edge (7)
      edge  [bend left=60] (8)
      edge [bend left=60] (9)
      edge [bend left=60] (10);
\path
(3) edge [bend right=60] (11) edge (12) edge (16);
\path
(7) edge [bend left=60] (12) edge (11) edge (14);
\path
(4) edge (8) edge (9) edge (10) edge [bend right=60] (13);
\path
(5) edge (8) edge (9) edge (10) edge (16);
\path
(6) edge (8) edge (9) edge (10) edge (12);
\path
(8) edge [bend left=60] (15);
\path
(9) edge (14);
\path
(10) edge (11);
\path
(13) edge (15) edge (16) edge (11);
\path
(14) edge (15) edge (16);
\path
(12) edge (15);

\end{tikzpicture}
\caption{A minimal non-semi-transitive subgraph of $K(8,3)$}\label{K83-fig}
\end{center}
\end{figure}
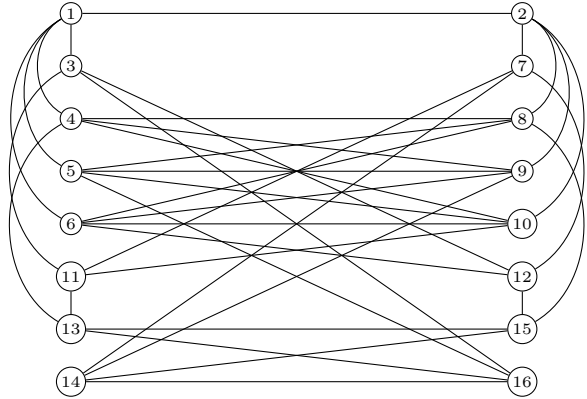

Determining if a triangle-free graph is semi-transitive is an NP-hard problem \cite{HKP16}. The existence of non-semi-transitive triangle-free graphs has been established via Erd\H{o}s' theorem \cite{Erdos} by Halld\'{o}rsson and the authors \cite{HKP11} in 2011 (also see \cite[Section 4.4]{KL15}). However, no explicit examples of such graphs were known until recent work of the first author and Saito \cite{KS19} who have shown {\em computationally} (using the user-friendly freely available software \cite{G}) that a certain subgraph on 16 vertices and 36 edges of the triangle-free Kneser graph $K(8,3)$ is not semi-transitive; the subgraph is shown in Fig.~\ref{K83-fig}. Thus, $K(8,3)$ itself on 56 vertices and 280 edges is non-semi-transitive. The question on the existence of smaller triangle-free non-semi-transitive  graphs has been raised in~\cite{KS19}. 

In Section~\ref{sec2} we prove that the Gr\"otzsch graph in Fig.~\ref{Grotzsch-graph} on 11 vertices 
is a smallest (by the number of vertices) non-semi-transitive triangle-free graph, and that the Chv\'atal graph in  Fig.~\ref{Chvatal-graph} is the smallest triangle-free 4-regular non-semi-transitive graph. 
In Section~\ref{sec3} we address the question on the existence of triangle-free semi-transitive graphs with chromatic number 4, and prove, in particular, that
Toft's graphs and the circulant graph $C(13;1,5)$ (the same as the Toeplitz graph $T_{13}(1,5,8,12)$) are such graphs. 
Finally, in Section~\ref{sec4} we discuss some open problems. 

\section{Non-semi-transitive orientability of the Gr\"otzsch graph and the Chv\'atal graph}\label{sec2}

The leftmost graph in Fig.~\ref{Grotzsch-graph} is the well-known Gr\"otzsch graph (also known as Mycielski graph).  It is well-known \cite{Ch} and is easy to prove that this graph is a minimal 4-chromatic triangle-free graph (and the only such graph on 11 vertices).

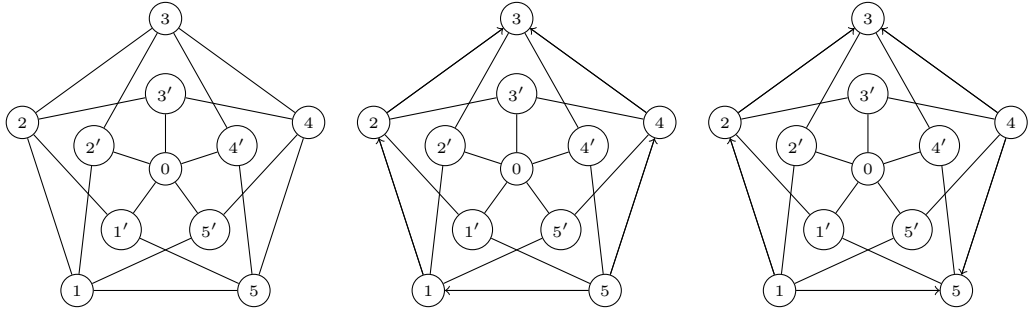
\begin{figure}
\begin{center}
\begin{tiny}

\begin{tabular}{ccc}

\begin{tikzpicture}[every node/.style={draw,circle}]
  \graph [clockwise,math nodes] {     
    subgraph C [V={ {3}, {4}, {5}, {1}, {2} }, name=A, radius=2cm]; 
    subgraph N [V={ {3'}, {4'}, {5'}, {1'}, {2'} }, name=B, radius=1cm];
    subgraph W [V={ {0} }, name=C, radius=0cm];
A 1 -- B 2;
A 1 -- B 5;
A 2 -- B 1;
A 2 -- B 3;
A 3 -- B 2;
A 3 -- B 4;
A 4 -- B 3;
A 4 -- B 5;
A 5 -- B 1;
A 5 -- B 4;

\foreach \i in {1,...,5}{
      C 1 -- B \i;  
    }

  }; 
\end{tikzpicture}

& 

\begin{tikzpicture}[every node/.style={draw,circle}]
  \graph [clockwise,math nodes] {     
    subgraph C [V={ {3}, {4}, {5}, {1}, {2} }, name=A, radius=2cm]; 
    subgraph N [V={ {3'}, {4'}, {5'}, {1'}, {2'} }, name=B, radius=1cm];
    subgraph W [V={ {0} }, name=C, radius=0cm];
A 1 -- B 2;
A 1 -- B 5;
A 2 -- B 1;
A 2 -- B 3;
A 3 -- B 2;
A 3 -- B 4;
A 4 -- B 3;
A 4 -- B 5;
A 5 -- B 1;
A 5 -- B 4;

A 2 -> A 1;
A 3 -> A 2;
A 3 -> A 4;
A 5 -> A 1; 
A 4 -> A 5;

\foreach \i in {1,...,5}{
      C 1 -- B \i;  
    }

  }; 
\end{tikzpicture}

& 

\begin{tikzpicture}[every node/.style={draw,circle}]
  \graph [clockwise,math nodes] {     
    subgraph C [V={ {3}, {4}, {5}, {1}, {2} }, name=A, radius=2cm]; 
    subgraph N [V={ {3'}, {4'}, {5'}, {1'}, {2'} }, name=B, radius=1cm];
    subgraph W [V={ {0} }, name=C, radius=0cm];
A 1 -- B 2;
A 1 -- B 5;
A 2 -- B 1;
A 2 -- B 3;
A 3 -- B 2;
A 3 -- B 4;
A 4 -- B 3;
A 4 -- B 5;
A 5 -- B 1;
A 5 -- B 4;

A 2 -> A 1;

A 2 -> A 3;

A 4 -> A 3;

A 5 -> A 1; 
A 4 -> A 5;

\foreach \i in {1,...,5}{
      C 1 -- B \i;  
    }

  }; 
\end{tikzpicture}

\end{tabular}

\end{tiny}
\end{center}
\vspace{-5mm}
\caption{The Gr\"otzsch graph and two of its partial orientations}\label{Grotzsch-graph}
\end{figure}

\begin{theorem}\label{Knes-K62-thm} The Gr\"otzsch graph $G$ is a smallest (by the number of vertices) non-semi-transitive graph. \end{theorem}

\begin{proof} To obtain a contradiction, suppose that $G$ is oriented semi-transitively. Then, the outer cycle formed by the vertices 1--5 either contains a directed path of length 3, or the longest directed path formed by the vertices is of length 2. Thus, we have two cases to consider. \\[-3mm]

\noindent
{\bf Case 1.} Taking into account symmetries, without loss of generality we can assume that $5 \rightarrow 1 \rightarrow 2 \rightarrow 3$ is a path of length 3, so that the orientation of the remaining two arcs must be $5 \rightarrow 4 \rightarrow 3$ by Lemma~\ref{lemma} as shown in the middle graph in Fig.~\ref{Grotzsch-graph}. Moreover, Lemma~\ref{lemma} can be used to complete orientations of the subgraphs induced by the vertices in the sets $\{1, 2, 3, 2'\}$, $\{1, 2, 1', 5\}$ and $\{3, 4, 5, 4'\}$, as shown in the left graph in Fig.~\ref{Grotzsch-graph-2}. We consider two subcases here depending on orientation of the arc $02'$.\\[-3mm]

\noindent
{\bf Case 1a.} Suppose $0 \rightarrow 2'$ is an arc. By Lemma~\ref{lemma}, 
\begin{itemize}
\item from the subgraph induced by  $0, 2', 3, 4'$, we have $0\rightarrow 4'$;
\item from the subgraph induced by  $0, 1', 5, 4'$, we have $0\rightarrow 1'$;
\item from the subgraph induced by  $0, 1', 2, 3'$, we have $0\rightarrow 3'$ and $3' \rightarrow 2$;
\item from the subgraph induced by  $2, 3, 4, 3'$, we have $3' \rightarrow 4$;
\item from the subgraph induced by  $0, 3', 4, 5'$, we have $0\rightarrow 5'$ and $5' \rightarrow 4$.
\end{itemize}
Now if $5'\rightarrow 1$ were an arc, the subgraph induced by $0, 5', 1, 2'$ would be a shortcut, while if  $1\rightarrow 5'$ were an arc, the subgraph induced by $1, 5', 4, 5$ would be a shortcut; a contradiction. \\[-3mm]

\noindent
{\bf Case 1b.} Suppose $2' \rightarrow 0$ is an arc. By Lemma~\ref{lemma}, 
\begin{itemize}
\item from the subgraph induced by  $0, 5', 1, 2'$, we have $1\rightarrow 5'$ and $5' \rightarrow 0$;
\item from the subgraph induced by  $1, 5, 4, 5'$, we have $4 \rightarrow 5'$;
\item from the subgraph induced by  $0, 3', 4, 5'$, we have $4\rightarrow 3'$ and $3' \rightarrow 0$;
\item from the subgraph induced by  $2, 3, 4, 3'$, we have $2 \rightarrow 3'$.
\end{itemize}
The contradiction is now obtained by the fact that there is no way to orient the arc $0\rightarrow 1'$ in the subgraph formed by $0, 1', 2, 3'$ without creating a cycle or a shortcut. \\[-3mm]

\begin{figure}
\begin{center}
\begin{tiny}

\begin{tabular}{cc}

\begin{tikzpicture}[every node/.style={draw,circle}]
  \graph [clockwise,math nodes] {     
    subgraph C [V={ {3}, {4}, {5}, {1}, {2} }, name=A, radius=2cm]; 
    subgraph N [V={ {3'}, {4'}, {5'}, {1'}, {2'} }, name=B, radius=1cm];
    subgraph W [V={ {0} }, name=C, radius=0cm];

A 3 -> B 4 -> A 5;
A 3 -> B 2 -> A 1;
A 4 -> B 5 -> A 1;

A 1 -- B 2;
A 1 -- B 5;
A 2 -- B 1;
A 2 -- B 3;
A 3 -- B 2;
A 3 -- B 4;
A 4 -- B 3;
A 4 -- B 5;
A 5 -- B 1;
A 5 -- B 4;

A 2 -> A 1;
A 3 -> A 2;
A 3 -> A 4;
A 5 -> A 1; 
A 4 -> A 5;

\foreach \i in {1,...,5}{
      C 1 -- B \i;  
    }

  }; 
\end{tikzpicture}

& 

\begin{tikzpicture}[every node/.style={draw,circle}]
  \graph [clockwise,math nodes] {     
    subgraph C [V={ {3}, {4}, {5}, {1}, {2} }, name=A, radius=2cm]; 
    subgraph N [V={ {3'}, {4'}, {5'}, {1'}, {2'} }, name=B, radius=1cm];
    subgraph W [V={ {0} }, name=C, radius=0cm];

A 4 -> B 5 -> A 1;

A 1 -- B 2;
A 1 -- B 5;
A 2 -- B 1;
A 2 -- B 3;
A 3 -- B 2;
A 3 -- B 4;
A 4 -- B 3;
A 4 -- B 5;
A 5 -- B 1;
A 5 -- B 4;

A 2 -> A 1;

A 2 -> A 3;

A 4 -> A 3;

A 5 -> A 1; 
A 4 -> A 5;

\foreach \i in {1,...,5}{
      C 1 -- B \i;  
    }

  }; 
\end{tikzpicture}

\end{tabular}

\end{tiny}
\end{center}
\vspace{-5mm}
\caption{Two partial orientations of the Gr\"otzsch graph}\label{Grotzsch-graph-2}
\end{figure}

\noindent
{\bf Case 2.} If the longest directed path induced by the vertices 1--5 is of length 2 then, again using the symmetries, we can assume the following orientation of the arcs: $1\rightarrow 2\rightarrow 3$, $1\rightarrow 5$, $4\rightarrow 5$ and $4\rightarrow 3$ as shown in the rightmost graph in Fig.~\ref{Grotzsch-graph}. Moreover, Lemma~\ref{lemma} can be used to complete orientations of the subgraph induced by the vertices in  $\{1, 2, 3, 2'\}$, as shown in the right graph in Fig.~\ref{Grotzsch-graph-2}. We consider two subcases here depending on orientation of the arc $02'$.\\[-3mm]

\noindent
{\bf Case 2a.} Suppose $0 \rightarrow 2'$ is an arc. By Lemma~\ref{lemma}, 

\begin{itemize}
\item from the subgraph induced by  $0, 2', 3, 4'$, we have $0\rightarrow 4'$ and $4' \rightarrow 3$;
\item from the subgraph induced by  $3, 4, 5, 4'$, we have $4'\rightarrow 5$;
\item from the subgraph induced by  $0, 1', 5, 4'$, we have $0\rightarrow 1'$ and $1' \rightarrow 5$;
\item from the subgraph induced by  $1, 2, 1', 5$, we have $1' \rightarrow 2$;
\item from the subgraph induced by  $0, 1', 2, 3'$, we have $0\rightarrow 3'$ and $3' \rightarrow 2$.
\item from the subgraph induced by  $2, 3, 4, 3'$, we have $3' \rightarrow 4$;
\item from the subgraph induced by  $0, 3', 4, 5'$, we have $0\rightarrow 5'$ and $5' \rightarrow 4$.
\end{itemize}
Now if $5'\rightarrow 1$ were an arc, the subgraph induced by $0, 5', 1, 2'$ would be a shortcut, while if  $1\rightarrow 5'$ were an arc, the subgraph induced by $1, 5', 4, 5$ would be a shortcut. A contradiction. \\[-3mm]

\noindent
{\bf Case 2b.} Suppose $2' \rightarrow 0$ is an arc. By Lemma~\ref{lemma},  
\begin{itemize}
\item from the subgraph induced by  $0, 5', 1, 2'$, we have $1\rightarrow 5'$ and $5' \rightarrow 0$;
\item from the subgraph induced by  $1, 5, 4, 5'$, we have $4 \rightarrow 5'$;
\item from the subgraph induced by  $0, 3', 4, 5'$, we have $4\rightarrow 3'$ and $3' \rightarrow 0$;
\item from the subgraph induced by  $2, 3, 4, 3'$, we have $2 \rightarrow 3'$;
\item from the subgraph induced by  $0, 1', 2, 3'$, we have $2\rightarrow 1'$ and $1' \rightarrow 0$;
\item from the subgraph induced by  $1, 2, 1', 5$, we have $5 \rightarrow 1'$;
\item from the subgraph induced by  $0, 1', 5, 4'$, we have $5\rightarrow 4'$ and $4' \rightarrow 0$.
\end{itemize}
Now if $3\rightarrow 4'$ were an arc, the subgraph induced by $2', 3, 4', 0$ would be a shortcut, while if  $4'\rightarrow 3$ were an arc, the subgraph induced by $4, 5, 4', 3$ would be a shortcut; a contradiction. \\[-3mm]

Thus, $G$ is not semi-transitive, and its minimality follows from the above mentioned fact that all triangle-free graphs on 10 or fewer vertices are 3-colorable, and thus semi-transitive by Theorem~\ref{3-color-graphs}.
\end{proof}

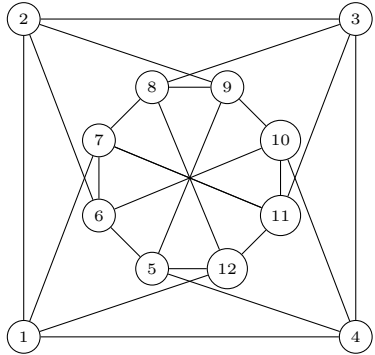
\begin{figure}
\begin{center}

\begin{tikzpicture}[node distance=1cm,auto,main node/.style={circle,draw,inner sep=2.5pt,minimum size=4pt}]

\node[main node] (9) {{\tiny 9}};
\node[main node] (8) [left of=9] {{\tiny 8}};
\node[main node] (10) [below right of=9] {{\tiny 10}};
\node[main node] (7) [below left of=8] {{\tiny 7}};
\node[main node] (6) [below of=7] {{\tiny 6}};
\node[main node] (5) [below right of=6] {{\tiny 5}};
\node[main node] (11) [below of=10] {{\tiny 11}};
\node[main node] (12) [below left of=11] {{\tiny 12}};
\node[main node] (2) [above left of=8,xshift=-1cm,yshift = 0.2cm] {{\tiny 2}};
\node[main node] (3) [above right of=9,xshift=1cm,yshift = 0.2cm] {{\tiny 3}};
\node[main node] (1) [below left of=5,xshift=-1cm,yshift = -0.2cm] {{\tiny 1}};
\node[main node] (4) [below right of=12,xshift=1cm,yshift = -0.2cm] {{\tiny 4}};


\path
(5) edge (6)
     edge (12)
     edge (9);
\path
(7) edge (6)
     edge (8)
     edge (11);
\path
(9) edge (8)
     edge (10);
\path
(11) edge (10)
     edge (12)
     edge (7);
\path
(8) edge (12);
\path
(6) edge (10);

\path
(1) edge (2)
     edge (4)
     edge (12)
     edge (7);

\path
(3) edge (2)
     edge (4)
     edge (11)
     edge (8);
\path
(2) edge (6)
     edge (9);

\path
(4) edge (5)
     edge (10);

\end{tikzpicture}

\caption{The Chv\'atal graph}\label{Chvatal-graph}
\end{center}
\end{figure}

The well-known Chv\'atal graph is presented in Fig.~\ref{Chvatal-graph}. It is the minimal 4-regular triangle-free 4-chromatic graph \cite{Ch}. Using the software \cite{G}, we found out that the Chv\'atal graph is not semi-transitive. We have also found an analytical proof of this fact via a long and tedious case analysis. Even being written using a specially developed short notation introduced in \cite{AKM15}, the proof takes several pages; therefore, we put the proof of our next theorem in Appendix for the most patient and interested Reader.

\begin{theorem}\label{Chv-thm} The Chv\'atal graph $H$ is a minimal $4$-regular triangle-free non-semi-transitive graph. \end{theorem}

As it was shown in \cite{Ch}, the Chv\'atal graph $H$ is not 4-critical: it remains 4-chromatic after removal of the edge $56$  (a graph is called 4-{\em critical}, if it is 4-chromatic, but removal of any edge makes it 3-chromatic). The software \cite{G} shows that the graph $H\setminus  \{56\}$ is still non-semi-transitive.  A proof of this fact is very similar to the proof of Theorem~\ref{Chv-thm}, in particular, it is also tedious, long and does not bring any new insights, so we omit it.

Note also, that proving that a graph $G$ is not semi-transitive immediately implies that the whole class of graphs containing $G$ as an induced subgraph is not semi-transitive; so, Theorems~\ref{Knes-K62-thm} and~\ref{Chvatal-graph} indeed give two classes of non-semi-transitive graphs.

\section{Semi-transitive triangle-free 4-chromatic graphs}\label{sec3}

As a matter of fact, no explicit examples of semi-transitively orientable triangle-free graphs with chromatic number 4, or larger, have been published yet.  However, as it was shown in \cite{HKP11}, the existence of such graphs easily follows from two well-known classical results presented below.

\begin{theorem}[\cite{Vitaver}]\label{Vit}  A graph is $k$-chromatic if and only if the minimum possible length of the longest directed path among all its acyclic orientations is $k-1$. 
\end{theorem}

\begin{theorem}[\cite{Erdos}]\label{Erd} For every $k\ge 2$ and $g\ge 3$ there exists a $k$-chromatic graph of girth $g$. 
\end{theorem}

Indeed, Theorem~\ref{Vit} implies that every graph whose girth is larger than its chromatic number has a semi-transitive orientation (as there is no chance for a shortcut in an acyclic orientation of such a graph), and Theorem~\ref{Erd} claims that such graphs exist. 
However, the existence of 4-chromatic semi-transitive graphs of girth $4$ does not follow from Theorems~\ref{Vit} and~\ref{Erd}. Below we present two explicit examples of such graphs.

\subsection{Circulant graphs}

 A {\em circulant graph} $C(n; a_1,\ldots, a_k)$ is a graph with the vertex set $\{0,\ldots , n-1\}$ and an edge set 
 $$E=\{ij \ |\ (i-j)\pmod n \mbox{ or } (j-i)\pmod n \mbox{ are in } \{a_1,\ldots, a_k\}\}.$$
 According to \cite{BT}, such graphs were first studied in 1932 by Foster, and the name comes from circulant matrices introduced by Catalan in 1846.
Circulant graphs have applications in distributed computer networks \cite{BCH}. 
Note that circulant graphs are indeed the Cayley graphs on cyclic groups $Z_n$; so, they are 
vertex-transitive (i.e.\ for every pair of its vertices there is an automorphism mapping one of them into another).  Circulant graphs are also a particular case of Toeplitz graphs \cite{Ghorban}. Various results on semi-transitivity of Toeplitz graphs have been obtained in \cite{CKKK}.

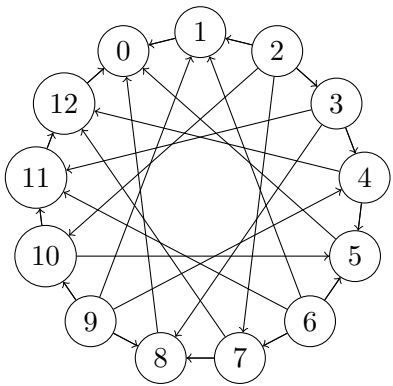
\begin{figure}
\begin{center}
\begin{tikzpicture}[every node/.style={draw,circle}]
  \graph [clockwise,math nodes] {     
    subgraph C [V={ {1}, {2}, {3}, {4}, {5}, {6}, {7}, {8}, {9}, {10}, {11}, {12}, {0} }, name=A, radius=2.2cm]; 

A 1 -> A 13;
A 2 -> A 1; A 2 -> A 3; A 2 -> A 7; A 2 -> A 10;
A 3 -> A 4; A 3 -> A 8; A 3 -> A 11; 
A 4 -> A 5; A 4 -> A 12;
A 5 -> A 13;
A 6 -> A 1; A 6 -> A 5; A 6 -> A 7; A 6 -> A 11;
A 7 -> A 8; A 7 -> A 12;
A 8 -> A 13;
A 9 -> A 1; A 9 -> A 4; A 9 -> A 8; A 9 -> A 10;
A 10 -> A 5; A 10 -> A 11;
A 11 -> A 12;
A 12 -> A 13;

  }; 
\end{tikzpicture}
\caption{A semi-transitive orientation of the circulant graph $C(13;1,5)$}\label{Toeplitz-13-1-5-8-12}
\end{center}
\end{figure}
It is well-known that the circulant graph $C(13;1,5)$ (which is the same as the Toeplitz graph $T_{13}(1,5,8,12)$) is the smallest vertex-transitive 4-chromatic triangle-free graph \cite{JT}. Of course, it would be nice to add this graph to our collection of minimal non-semi-transitive 4-chromatic triangle-free graphs in the previous section, but the graph appears to be semi-transitive, as follows from the next theorem.

\begin{theorem}\label{Circ} The circulant graph $C(13;1,5)$ is a $4$-chromatic $4$-regular semi-transitive graph of girth $4$. 
\end{theorem}

\begin{proof} Let $G:=C(13;1,5)$ and consider its orientation presented in Fig.~\ref{Toeplitz-13-1-5-8-12}. It is easy to verify 
by successive deletion of sources and/or sinks that this orientation is acyclic.
The following two easy observations help in checking the absence of shortcuts.\\[-3mm] 

\noindent {\bf Claim $1$.} If $v$ is a source or a sink in a directed graph and either all its neighbors are sinks in $G\setminus v$ or all of them are sources in $G\setminus v$ then $v$ does not lie in any 
shortcut.

Indeed, assume $v$ lies in a shortcut with a long path $v_0\rightarrow v_1\rightarrow \cdots \rightarrow v_{k-1}\rightarrow v_k$. If $v$ is a sink then $v=v_k$, and thus, $v_{k-1}$ cannot be a source in $G\setminus v$ and $v_0$ cannot be a sink in $G\setminus v$. If $v$ is a source then $v=v_0$, and thus, $v_k$ cannot be a source in $G\setminus v$ and $v_1$ cannot be a sink in $G\setminus v$. \\[-3mm] 

\noindent {\bf Claim $2$.} If $v$ is a source that lies in a shortcut, then there are two directed paths $P_0,P_1$ starting at $v$ so that $P_0$ starts with a shortcutting arc $u\rightarrow v$ and $v$ is $k$-th vertex in $P_1$ for some $k\ge 4$. 

This claim follows directly from the definition of the shortcut.

By Claim 1, $0$ is not a part of any shortcut in $G$, and $1$ does not lie in a shortcut in $G\setminus 0$. In the graph $G\setminus \{0,1\}$ the paths starting in $6$ are $\{65, 678, 67(12),6(11)(12)\}$, and the paths starting in $9$ are $$\{9(10)5, 9(10)(11)(12),945, 94(12),98 \}.$$ By Claim~2, both these vertices are not in shortcuts. Applying Claim~1 to $G\setminus \{ 0,1,6,9\}$, remove successively the vertices $2$, $7$, and $8$. In the obtained graph, exclude $10$ 
by Claim~2 (the only paths are $(10)5$ and $(10)(11)(12)$), and afterwards, remove $5$ and $12$ by Claim~1. The remaining graph on the vertex set $\{3,4,11\}$ is a tree. 

So, there are no shortcuts in $G$ and the considered orientation is semi-transitive.
\end{proof} 

As it was proved in \cite{Heu}, a connected $4$-regular circulant other than $C(13;1,5)$  has chromatic number $4$ if and only if it is isomorphic to the circulant graph $C(n;1,2)$ for some $n=3t+1$ or $n=3t+2$ where $t\ge 2$. Although all such circulants contain triangles, we would like to close the question on the semi-transitivity of 4-regular circulants by proving the following result.

\begin{theorem}\label{Circ2} Each $4$-regular circulant graph is semi-transitive. 
\end{theorem}

Clearly, a disjoint union of semi-transitive graphs is semi-transitive, $K_5=C(5;1,2)$ admits transitive orientation and every 3-colorable graph is semi-transitive by Theorem~\ref{3-color-graphs}. Hence, Theorem~\ref{Circ2} is a direct corollary of the above mentioned result in \cite{Heu},
Theorem~\ref{Circ} and the following lemma.

\begin{lemma}\label{Circs} A circulant graph $C(n;1,2)$ is semi-transitive for each $n\ge 6$. 
\end{lemma}

\begin{proof}
Consider a circulant graph $G=C(n;1,2)$ with the vertex set $V=\{0,1,\ldots, n-1\}$. Orient the edges of the subgraph induced by the subset $V_0=\{0,1,\ldots, n-3\}$ from lowest to highest (i.e. $0\rightarrow 1$, $0\rightarrow2$, $1\rightarrow2$, $1\rightarrow 3$, etc) and set the orientation of the remaining seven edges as follows:  
$1\rightarrow n-1,\ 0\rightarrow n-1,\ 0\rightarrow n-2,\ n-2\rightarrow n-4,\ n-2\rightarrow n-3,\ n-2\rightarrow n-1,\ n-1\rightarrow n-3.$
It is easy to see that the orientation is acyclic. Assume that there is a shortcut $v_0\rightarrow \cdots \rightarrow v_k$ with a shortcutting arc 
$v_0\rightarrow v_k$ where $k\ge 3$. Clearly, the shortcut cannot lie in $V_0$ since otherwise for the shortcutting arc we have $k\le 2$ by the definition of the circulant, a contradiction with $k\ge 3$. So, $n-2$ or $n-1$ must be in the shortcut. By symmetry, we may assume that the shortcut contains $n-2$ (otherwise, reverse all arcs and swap $n-1$ with $n-2$ and $i$ with $n-3-i$ for all $i=0,\ldots,n-3$). Since the longest path outgoing from $n-2$ has length $2$, $v_0\ne n-2$. But then the shortcut must contain the arc $0\rightarrow n-2$. Since $0$ is a source, we have $v_0=0,\ v_1=n-2$. There are only two paths of length at least $3$ starting with the arc $0\rightarrow n-2$ (namely, $0 \rightarrow n-2 \rightarrow n-1 \rightarrow n-3$ and $0 \rightarrow n-2 \rightarrow n-4 \rightarrow n-3$). But in both cases $G$ does not contain the shortcutting arc $0\rightarrow n-3$. So, the presented orientation is semi-transitive.
\end{proof}

\begin{remark} If $n=5$ then the orientation in Lemma~\ref{Circs} provides a transitive orientation of $K_5=C(5;1,2)$. \end{remark}

\subsection{Toft's graphs}

Another nice example of 4-chromatic semi-transitive graphs of girth 4 is given by 
Toft's graphs $T_n$ that were introduced in \cite{Toft} as first instances of dense 4-critical graphs (see \cite{Peg} for various constructions of dense critical graphs).

Let $n>3$ be odd. The construction of Toft's graph $T_n$ is as follows. It has a vertex set $V=A_1\cup A_2\cup A_3\cup A_4$ of $4n$ vertices where $A_1$ and $A_4$ induce odd cycles $C_n$ and $A_2\cup A_3$ induces  the complete bipartite graph $K_{n,n}$ with parts $A_2$ and $A_3$. There is also a perfect matching whose all edges connect either $A_1$ with $A_2$ or $A_3$ with $A_4$.

\begin{theorem} Toft's graph $T_n$ is semi-transitive.
\end{theorem}

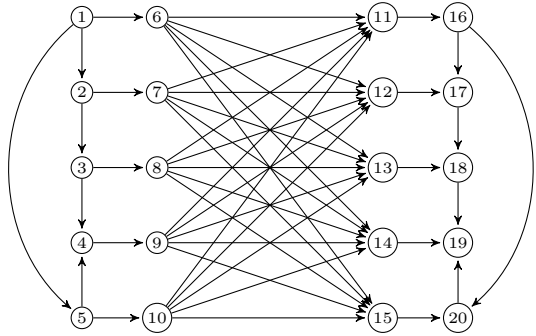
\begin{figure}
\begin{center}
\begin{tikzpicture}[->, >=stealth', shorten >=1pt, node distance=1cm,auto,main 
node/.style={circle,draw,inner sep=1pt,minimum size=2pt}]

\node[main node] (1) {{\tiny 1}};
\node[main node] (2) [below of=1] {{\tiny 2}};
\node[main node] (3) [below of=2] {{\tiny 3}};
\node[main node] (4) [below of=3] {{\tiny 4}};
\node[main node] (5) [below of=4] {{\tiny 5}};

\node[main node] (6) [right of=1]  {{\tiny 6}};
\node[main node] (7) [below of=6] {{\tiny 7}};
\node[main node] (8) [below of=7] {{\tiny 8}};
\node[main node] (9) [below of=8] {{\tiny 9}};
\node[main node] (10) [below of=9] {{\tiny 10}};

\node[main node] (11) [right of=6, xshift=2cm]  {{\tiny 11}};
\node[main node] (12) [below of=11] {{\tiny 12}};
\node[main node] (13) [below of=12] {{\tiny 13}};
\node[main node] (14) [below of=13] {{\tiny 14}};
\node[main node] (15) [below of=14] {{\tiny 15}};

\node[main node] (16) [right of=11]  {{\tiny 16}};
\node[main node] (17) [below of=16] {{\tiny 17}};
\node[main node] (18) [below of=17] {{\tiny 18}};
\node[main node] (19) [below of=18] {{\tiny 19}};
\node[main node] (20) [below of=19] {{\tiny 20}};

\path (1) edge (2) 
              edge (6) 
              edge [bend right=50] (5);
\path (2) edge (3) 
              edge (7); 
\path (3) edge (4) 
              edge (8); 
\path (5) edge (4);
\path (4) edge (9); 
\path (5) edge (10); 
\path (11) edge (16);
\path (16)  edge (17) 
              edge [bend left=50] (20);
\path (12) edge (17);
\path (17) edge (18); 
\path (13) edge (18);
\path (18)  edge (19); 
\path (14) edge (19);
\path (20)  edge (19); 
\path (15) edge (20); 
\path (6) edge (11) 
              edge (12)
              edge (13)
              edge (14)
              edge (15); 
\path (7) edge (11) 
              edge (12)
              edge (13)
              edge (14)
              edge (15); 
\path (8) edge (11) 
              edge (12)
              edge (13)
              edge (14)
              edge (15); 
\path (9) edge (11) 
              edge (12)
              edge (13)
              edge (14)
              edge (15); 
\path (10) edge (11) 
              edge (12)
              edge (13)
              edge (14)
              edge (15); 

\end{tikzpicture}
\caption{A semi-transitive orientation of Toft's graph $T_5$}\label{Toft-T5-sem}
\end{center}
\end{figure}

\begin{proof} A semi-transitive orientation of $T_n$ can be constructed as follows. Every arc $uv$ where $u\in A_i$ and $v\in A_{i+1}$ for any $i\in \{1,2,3\}$ is directed $u\rightarrow v$. 
The cycles $A_1$ and $A_4$ are oriented semi-transitively in an arbitrary way (e.~g. by  arranging in each of them two disjoint directed paths of lengths $2$ and $n-2$ starting in a same node). An example of Toft's graph $T_5$ and its orientation is shown in Fig.~\ref{Toft-T5-sem}.

Clearly, this orientation is acyclic. Assume, there is a shortcut $C$  with a long path $v_0\rightarrow 
\cdots \rightarrow v_k$. Then either $v_0, v_k\in A_i$ for some $i\in \{1,2,3,4\}$ or $v_0\in A_i, v_k\in A_{i+1}$ for some $i\in \{1,2,3\}$. The first case is impossible since the sets $A_2$ and $A_3$ are independent and the orientations of $A_1$ and $A_4$ are semi-transitive. The second case cannot occur since all vertices form $A_2$ and $A_3$ have degree $1$ in the subgraphs induced by $A_1\cup A_2$ and $A_3\cup A_4$, respectively, and the subgraph induced by $A_2\cup A_3$ has no directed paths of length more than $1$. Therefore, the presented orientation is semi-transitive.
\end{proof}

\section{Open problems}\label{sec4}

In this paper we presented examples of non-semi-transitive triangle-free graphs of girth 4, namely the Gr\"otzsch graph, the Chv\'atal graph, and the Chv\'atal graph without certain edge. However, for higher girths the similar existence question is still open.

\begin{problem}\label{prob1} Do there exist non-semi-transitive graphs of girth $g$ for every $g\geq 5$? \end{problem}

We also presented examples of semi-transitive $k$-chromatic graphs of girth $k$ for $k=4$. Finding similar explicit instances could be of interest for larger $k$, especially in terms of minimality according to different criteria.

\begin{problem}\label{prob2} Do there exist semi-transitive $k$-chromatic graphs of girth $k$ for every $k\geq 5$? If yes, then are there regular or vertex-transitive examples? What is the minimum number of vertices and/or edges in such graphs? How dense can they be?\end{problem}

Problem~\ref{prob2} is some kind of a complement question to Problem~\ref{prob1}, so at least one of these problems must have a positive answer. However, we conjecture that the answer is positive for both of them.

Finally, it would be interesting to extend the results of Lemma~\ref{Circs}. Note, that in general the circulants may be not semi-transitive. For instance, $C(14;1,3,4,5)$ is not \cite{G}. But is this true for $C(n;1,\ldots,k)$?

\begin{problem} Are all circulants $C(n;1,2,\ldots,k)$ semi-transitive? What about circulants $C(n; t, t+1,\ldots, k)$ for some integers $k$ and $t$ satisfying $k - t > 1$?\end{problem}
\medskip 

\noindent
{\bf Anknowledgements. } The authors are grateful to the unknown referees for their valuable comments and suggestions. The work of the second author was partially supported by the program of fundamental scientific researches of the SB RAS, project 0314--2019--0014.

\section*{Appendix. Proof of Theorem~\ref{Chv-thm}}

Our proof of non-semi-transitivity of the Chv\'atal graph uses symmetries and Lemma~\ref{lemma}. It results in considering 13 partially oriented copies of the graph that can be drawn to check our arguments. 
Note that non-semi-transitivity of the Chv\'atal graph can be easily checked using the software \cite{G}. To make text of the proof as short as  possible, we use the following brief notation introduced in \cite{AKM15}: 
\begin{itemize}
\item ``MC $X$'' means ``Move to (consider) the partially oriented copy $X$ (obtained earlier)'';
\item ``C$x_1\ldots x_k$'' stands for ``Apply Lemma~\ref{lemma} to a partially directed cycle $x_1\ldots x_k$ (and get some new arcs)'';
\item ``B$xy$ (NC $X$)'' denotes ``Branch on the arc $xy$: if it goes $y\rightarrow x$, create a copy $X$ (to be considered later); otherwise, put the orientation $x\rightarrow y$ and continue the analysis of the current copy'';
\item ``S$x_1\ldots x_k$'' means ``The vertices $x_1,\ldots,x_k$ induce a shortcut, a contradiction''. 
\end{itemize}

\begin{figure}
\begin{center}
\begin{tabular}{cc}
\begin{tikzpicture}[node distance=1cm,auto,main node/.style={circle,draw,inner sep=2.5pt,minimum size=4pt}]

\node[main node] (9) {{\tiny 9}};
\node[main node] (8) [left of=9] {{\tiny 8}};
\node[main node] (10) [below right of=9] {{\tiny 10}};
\node[main node] (7) [below left of=8] {{\tiny 7}};
\node[main node] (6) [below of=7] {{\tiny 6}};
\node[main node] (5) [below right of=6] {{\tiny 5}};
\node[main node] (11) [below of=10] {{\tiny 11}};
\node[main node] (12) [below left of=11] {{\tiny 12}};
\node[main node] (2) [above left of=8,xshift=-1cm,yshift = 0.2cm] {{\tiny 2}};
\node[main node] (3) [above right of=9,xshift=1cm,yshift = 0.2cm] {{\tiny 3}};
\node[main node] (1) [below left of=5,xshift=-1cm,yshift = -0.2cm] {{\tiny 1}};
\node[main node] (4) [below right of=12,xshift=1cm,yshift = -0.2cm] {{\tiny 4}};
\node (x) [below right of=1,xshift=1.4cm] {A};

\path
(2) [->,>=stealth', shorten >=1pt] edge (3);
\path
(3) [->,>=stealth', shorten >=1pt] edge (4);
\path
(4) [->,>=stealth', shorten >=1pt] edge (5);
\path
(4) [->,>=stealth', shorten >=1pt] edge (10);
\path
(2) [->,>=stealth', shorten >=1pt] edge (1);
\path
(1) [->,>=stealth', shorten >=1pt] edge (4);

\path
(5) edge (6)
     edge (12)
     edge (9);
\path
(7) edge (6)
     edge (8)
     edge (11);
\path
(9) edge (8)
     edge (10);
\path
(11) edge (10)
     edge (12)
     edge (7);
\path
(8) edge (12);
\path
(6) edge (10);

\path
(1) edge (2)
     edge (4)
     edge (12)
     edge (7);

\path
(3) edge (2)
     edge (4)
     edge (11)
     edge (8);
\path
(2) edge (6)
     edge (9);

\path
(4) edge (5)
     edge (10);

\end{tikzpicture}

& 

\begin{tikzpicture}[node distance=1cm,auto,main node/.style={circle,draw,inner sep=2.5pt,minimum size=4pt}]

\node[main node] (9) {{\tiny 9}};
\node[main node] (8) [left of=9] {{\tiny 8}};
\node[main node] (10) [below right of=9] {{\tiny 10}};
\node[main node] (7) [below left of=8] {{\tiny 7}};
\node[main node] (6) [below of=7] {{\tiny 6}};
\node[main node] (5) [below right of=6] {{\tiny 5}};
\node[main node] (11) [below of=10] {{\tiny 11}};
\node[main node] (12) [below left of=11] {{\tiny 12}};
\node[main node] (2) [above left of=8,xshift=-1cm,yshift = 0.2cm] {{\tiny 2}};
\node[main node] (3) [above right of=9,xshift=1cm,yshift = 0.2cm] {{\tiny 3}};
\node[main node] (1) [below left of=5,xshift=-1cm,yshift = -0.2cm] {{\tiny 1}};
\node[main node] (4) [below right of=12,xshift=1cm,yshift = -0.2cm] {{\tiny 4}};
\node (x) [below right of=1,xshift=1.4cm] {B};

\path
(2) [->,>=stealth', shorten >=1pt] edge (1);
\path
(2) [->,>=stealth', shorten >=1pt] edge (3);
\path
(3) [->,>=stealth', shorten >=1pt] edge (4);
\path
(1) [->,>=stealth', shorten >=1pt] edge (4);
\path
(4) [->,>=stealth', shorten >=1pt] edge (5);
\path
(10) [->,>=stealth', shorten >=1pt] edge (4);

\path
(5) edge (6)
     edge (12)
     edge (9);
\path
(7) edge (6)
     edge (8)
     edge (11);
\path
(9) edge (8)
     edge (10);
\path
(11) edge (10)
     edge (12)
     edge (7);
\path
(8) edge (12);
\path
(6) edge (10);

\path
(1) edge (2)
     edge (4)
     edge (12)
     edge (7);

\path
(3) edge (2)
     edge (4)
     edge (11)
     edge (8);
\path
(2) edge (6)
     edge (9);

\path
(4) edge (5)
     edge (10);

\end{tikzpicture}

\end{tabular}
\caption{Two partial orientations, $A$ and $B$, of the Chv\'atal graph}\label{Chvatal-graph-dir3}
\end{center}
\end{figure}
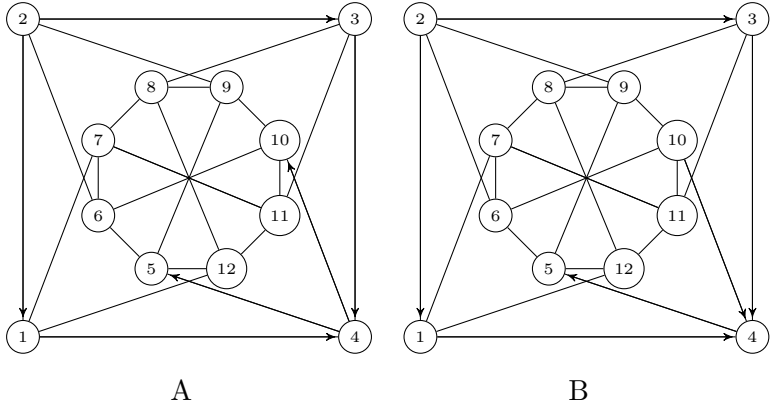

\noindent
For instance, the string  ``MC B, C1234, C23456, ..., S98(12)5;'' below means: ``Consider $B$ (in Fig.~\ref{Chvatal-graph-dir1}), by Lemma~\ref{lemma} in the cycle 1234 we must have $2\rightarrow 1$ and $1\rightarrow 4$, in the cycle 23456 we must have $2\rightarrow 6$ and $6\rightarrow 5$, ..., a contradiction is obtained with the cycle 98(12)5 being a shortcut''.

\begin{figure}
\begin{center}
\begin{tabular}{cc}
\begin{tikzpicture}[node distance=1cm,auto,main node/.style={circle,draw,inner sep=2.5pt,minimum size=4pt}]

\node[main node] (9) {{\tiny 9}};
\node[main node] (8) [left of=9] {{\tiny 8}};
\node[main node] (10) [below right of=9] {{\tiny 10}};
\node[main node] (7) [below left of=8] {{\tiny 7}};
\node[main node] (6) [below of=7] {{\tiny 6}};
\node[main node] (5) [below right of=6] {{\tiny 5}};
\node[main node] (11) [below of=10] {{\tiny 11}};
\node[main node] (12) [below left of=11] {{\tiny 12}};
\node[main node] (2) [above left of=8,xshift=-1cm,yshift = 0.2cm] {{\tiny 2}};
\node[main node] (3) [above right of=9,xshift=1cm,yshift = 0.2cm] {{\tiny 3}};
\node[main node] (1) [below left of=5,xshift=-1cm,yshift = -0.2cm] {{\tiny 1}};
\node[main node] (4) [below right of=12,xshift=1cm,yshift = -0.2cm] {{\tiny 4}};
\node (x) [below right of=1,xshift=1.4cm] {C};

\path
(2) [->,>=stealth', shorten >=1pt] edge (1);
\path
(2) [->,>=stealth', shorten >=1pt] edge (3);
\path
(3) [->,>=stealth', shorten >=1pt] edge (4);
\path
(1) [->,>=stealth', shorten >=1pt] edge (4);
\path
(5) [->,>=stealth', shorten >=1pt] edge (4);
\path
(10) [->,>=stealth', shorten >=1pt] edge (4);

\path
(5) edge (6)
     edge (12)
     edge (9);
\path
(7) edge (6)
     edge (8)
     edge (11);
\path
(9) edge (8)
     edge (10);
\path
(11) edge (10)
     edge (12)
     edge (7);
\path
(8) edge (12);
\path
(6) edge (10);

\path
(1) edge (2)
     edge (4)
     edge (12)
     edge (7);

\path
(3) edge (2)
     edge (4)
     edge (11)
     edge (8);
\path
(2) edge (6)
     edge (9);

\path
(4) edge (5)
     edge (10);

\end{tikzpicture}

& 

\begin{tikzpicture}[node distance=1cm,auto,main node/.style={circle,draw,inner sep=2.5pt,minimum size=4pt}]

\node[main node] (9) {{\tiny 9}};
\node[main node] (8) [left of=9] {{\tiny 8}};
\node[main node] (10) [below right of=9] {{\tiny 10}};
\node[main node] (7) [below left of=8] {{\tiny 7}};
\node[main node] (6) [below of=7] {{\tiny 6}};
\node[main node] (5) [below right of=6] {{\tiny 5}};
\node[main node] (11) [below of=10] {{\tiny 11}};
\node[main node] (12) [below left of=11] {{\tiny 12}};
\node[main node] (2) [above left of=8,xshift=-1cm,yshift = 0.2cm] {{\tiny 2}};
\node[main node] (3) [above right of=9,xshift=1cm,yshift = 0.2cm] {{\tiny 3}};
\node[main node] (1) [below left of=5,xshift=-1cm,yshift = -0.2cm] {{\tiny 1}};
\node[main node] (4) [below right of=12,xshift=1cm,yshift = -0.2cm] {{\tiny 4}};
\node (x) [below right of=1,xshift=1.4cm] {D};

\path
(2) [->,>=stealth', shorten >=1pt] edge (1);
\path
(2) [->,>=stealth', shorten >=1pt] edge (3);
\path
(4) [->,>=stealth', shorten >=1pt] edge (3);
\path
(4) [->,>=stealth', shorten >=1pt] edge (1);
\path
(5) [->,>=stealth', shorten >=1pt] edge (4);
\path
(10) [->,>=stealth', shorten >=1pt] edge (4);

\path
(5) edge (6)
     edge (12)
     edge (9);
\path
(7) edge (6)
     edge (8)
     edge (11);
\path
(9) edge (8)
     edge (10);
\path
(11) edge (10)
     edge (12)
     edge (7);
\path
(8) edge (12);
\path
(6) edge (10);

\path
(1) edge (2)
     edge (4)
     edge (12)
     edge (7);

\path
(3) edge (2)
     edge (4)
     edge (11)
     edge (8);
\path
(2) edge (6)
     edge (9);

\path
(4) edge (5)
     edge (10);

\end{tikzpicture}

\end{tabular}
\caption{Two partial orientations, $C$ and $D$, of the Chv\'atal graph}\label{Chvatal-graph-dir1}
\end{center}
\end{figure}
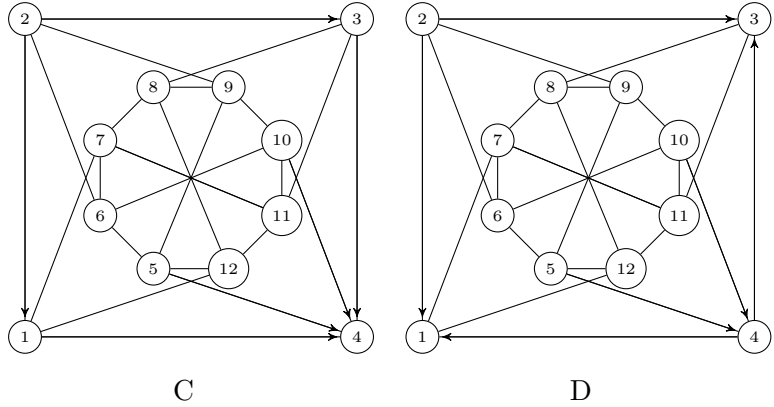

\begin{proof}[Proof of Theorem~\ref{Chv-thm}] Suppose that the Chv\'atal graph $H$ (see Fig.~\ref{Chvatal-graph}) can be oriented semi-transitively. 
By lemma~\ref{lemma}, exactly two arcs of the cycle 1234 are directed clockwise. So, by symmetry, we may assume that this cycle has either arcs $2\rightarrow 1, 2\rightarrow 3, 1\rightarrow 4$ and $3\rightarrow 4$ or arcs $2\rightarrow 1, 2\rightarrow 3, 4\rightarrow 1$ and $4\rightarrow 3$. We also branch on the orientation of the number of outgoing from the vertex $4$ edges among $45$ and $4(10)$: it can be 0, 1, or 2. The corresponding cases induce 6 initial copies (partial orientations) from $A$ to $F$, presented in in Fig.~\ref{Chvatal-graph-dir3}, Fig.~\ref{Chvatal-graph-dir1}, and Fig.~\ref{Chvatal-graph-dir2}.
Next we consider all of them starting from $A$ and using the notation introduced above. In each case, we will obtain a contradiction showing that $H$ cannot be oriented semi-transitively. \\

\begin{small}
\noindent
- MC $A$, C1234; due to the symmetry with respect to the diagonal 2--4, we can assume existence of the arc $8\rightarrow 7$;  C234(10)9, C145(12), C34(10)(11), C871(12), C2176, C67(11)(10), C387(11), C8(12)59, S2389;\\
-  MC $B$, C1234, C23456, C2659, C(10)456, C145(12), B67 (NC $B_1$), C(10)67(11), C(10)(11)34, C3(11)78, C2389, C1(12)87, S98(12)5; \\
- MC $B_1$, C2176, C71(12)8, C8(12)59, C2389, C783(11), C(11)34(10), S7(11)(10)6;\\ 
-  MC $C$, C1234; we can assume presence of the arcs $2\rightarrow 6$ and $2\rightarrow 9$ (otherwise, changing direction of all arcs results in a copy $A$ or $B$), and also presence of the edge $8\rightarrow 7$ (because of the symmetry with respect to the diagonal $2--4$); we branch on two arcs, 17 and 38, simultaneously: if $3\rightarrow 8$ and $1\rightarrow 7$ NC $C_1$; if $8\rightarrow 3$ and $7\rightarrow 1$ NC $C_2$; if $8\rightarrow 3$ and $1\rightarrow 7$ NC $C_3$; if $3\rightarrow 8$ and $7\rightarrow 1$ then S23871; \\
-  MC $C_1$, C2389, C387(11), C34(10)(11), C(10)(11)76, C(10)654, C2956, C598(12), C(12)871, S5(12)14;\\
-  MC $C_2$, C8371, C871(12), C5(12)14, C8(12)59, C954(10), C29(10)6, C76(10)(11), C87(11)3, S3(11)(10)4;\\
-  MC $C_3$, C2176, C2983, B9(10) (NC $C_4$), C9(11)45, C895(12), C8(12)17, S1(12)54;\\
-  MC $C_4$, C(10)954, C2659, C56(10)9, C(10)67(11), C(10)(11)34, S83(11)7; \\
-  MC $D$; using symmetry with respect to the diagonal 2--4 we can assume presence of the arc $8\rightarrow 7$; C(10)43(11),  C541(12), C(11)387, C(10)(11)76, B62 (NC $D_1$), C(10)629, C6217, C871(12), C9238, C98(12)5, S(10)954;\\
-  MC $D_1$, C2671, C(12)178, C5(12)89, C59(10)6, C(10)926, S2983;\\
-  MC $E$; we can assume that $6\rightarrow 2$ and $9\rightarrow 2$ are not arcs at the same time (otherwise we get the graph $D$); C(10)43(11), C(10)459, C(10)456; since the presence of  $6\rightarrow 2$ and $2\rightarrow 9$, or the presence of  $9\rightarrow 2$ and $2\rightarrow 6$ gives S(10)926, while the presence of $6\rightarrow 2$ and $9\rightarrow 2$ is forbidden above, we have $2\rightarrow 6$ and $2\rightarrow 9$; B83 (NC $E_1$), C2983, C895(12), C41(12)5, C8(12)17, C2671, C(10)(11)76, S87(11)3; \\
-  MC $E_1$, C(11)387, C(10)(11)76, C2671, C178(12), C95(12)8, S41(12)5; \\
-  MC $F$; note that the vertex 2 must be a source (the in-degree is 0), and 1 and 3 must be sinks (the out-degree is 0), since otherwise after renaming the vertices, and if necessary reversing the directions of all arcs, we would obtain $D$ or $E$; using symmetry with respect to the diagonal 2--4, we can assume $8\rightarrow 7$; C45(12)1, C2671, C87(12)3, C(11)76(10), C4(10)65, C2956, C(12)598, S(12)871. 
\end{small}

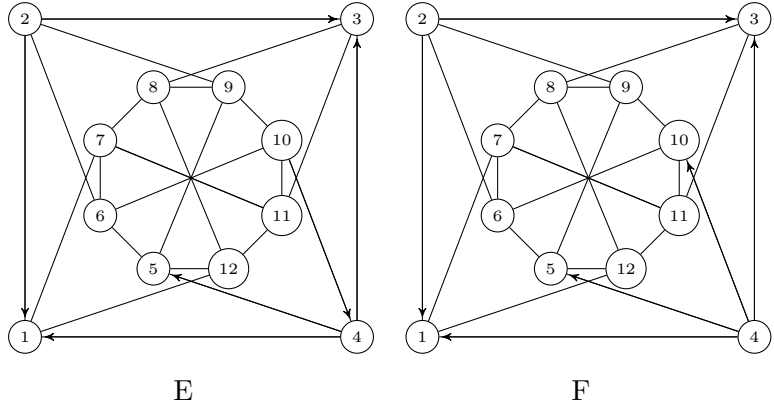
\begin{figure}
\begin{center}
\begin{tabular}{cc}

\begin{tikzpicture}[node distance=1cm,auto,main node/.style={circle,draw,inner sep=2.5pt,minimum size=4pt}]

\node[main node] (9) {{\tiny 9}};
\node[main node] (8) [left of=9] {{\tiny 8}};
\node[main node] (10) [below right of=9] {{\tiny 10}};
\node[main node] (7) [below left of=8] {{\tiny 7}};
\node[main node] (6) [below of=7] {{\tiny 6}};
\node[main node] (5) [below right of=6] {{\tiny 5}};
\node[main node] (11) [below of=10] {{\tiny 11}};
\node[main node] (12) [below left of=11] {{\tiny 12}};
\node[main node] (2) [above left of=8,xshift=-1cm,yshift = 0.2cm] {{\tiny 2}};
\node[main node] (3) [above right of=9,xshift=1cm,yshift = 0.2cm] {{\tiny 3}};
\node[main node] (1) [below left of=5,xshift=-1cm,yshift = -0.2cm] {{\tiny 1}};
\node[main node] (4) [below right of=12,xshift=1cm,yshift = -0.2cm] {{\tiny 4}};
\node (x) [below right of=1,xshift=1.4cm] {E};

\path
(10) [->,>=stealth', shorten >=1pt] edge (4);
\path
(4) [->,>=stealth', shorten >=1pt] edge (5);
\path
(2) [->,>=stealth', shorten >=1pt] edge (1);
\path
(2) [->,>=stealth', shorten >=1pt] edge (3);
\path
(4) [->,>=stealth', shorten >=1pt] edge (1);
\path
(4) [->,>=stealth', shorten >=1pt] edge (3);

\path
(5) edge (6)
     edge (12)
     edge (9);
\path
(7) edge (6)
     edge (8)
     edge (11);
\path
(9) edge (8)
     edge (10);
\path
(11) edge (10)
     edge (12)
     edge (7);
\path
(8) edge (12);
\path
(6) edge (10);

\path
(1) edge (2)
     edge (4)
     edge (12)
     edge (7);

\path
(3) edge (2)
     edge (4)
     edge (11)
     edge (8);
\path
(2) edge (6)
     edge (9);

\path
(4) edge (5)
     edge (10);

\end{tikzpicture}

& 

\begin{tikzpicture}[node distance=1cm,auto,main node/.style={circle,draw,inner sep=2.5pt,minimum size=4pt}]

\node[main node] (9) {{\tiny 9}};
\node[main node] (8) [left of=9] {{\tiny 8}};
\node[main node] (10) [below right of=9] {{\tiny 10}};
\node[main node] (7) [below left of=8] {{\tiny 7}};
\node[main node] (6) [below of=7] {{\tiny 6}};
\node[main node] (5) [below right of=6] {{\tiny 5}};
\node[main node] (11) [below of=10] {{\tiny 11}};
\node[main node] (12) [below left of=11] {{\tiny 12}};
\node[main node] (2) [above left of=8,xshift=-1cm,yshift = 0.2cm] {{\tiny 2}};
\node[main node] (3) [above right of=9,xshift=1cm,yshift = 0.2cm] {{\tiny 3}};
\node[main node] (1) [below left of=5,xshift=-1cm,yshift = -0.2cm] {{\tiny 1}};
\node[main node] (4) [below right of=12,xshift=1cm,yshift = -0.2cm] {{\tiny 4}};
\node (x) [below right of=1,xshift=1.4cm] {F};

\path
(2) [->,>=stealth', shorten >=1pt] edge (1);
\path
(2) [->,>=stealth', shorten >=1pt] edge (3);
\path
(4) [->,>=stealth', shorten >=1pt] edge (3);
\path
(4) [->,>=stealth', shorten >=1pt] edge (1);
\path
(4) [->,>=stealth', shorten >=1pt] edge (5);
\path
(4) [->,>=stealth', shorten >=1pt] edge (10);

\path
(5) edge (6)
     edge (12)
     edge (9);
\path
(7) edge (6)
     edge (8)
     edge (11);
\path
(9) edge (8)
     edge (10);
\path
(11) edge (10)
     edge (12)
     edge (7);
\path
(8) edge (12);
\path
(6) edge (10);

\path
(1) edge (2)
     edge (4)
     edge (12)
     edge (7);

\path
(3) edge (2)
     edge (4)
     edge (11)
     edge (8);
\path
(2) edge (6)
     edge (9);

\path
(4) edge (5)
     edge (10);

\end{tikzpicture}

\end{tabular}
\caption{Two partial orientations, $E$ and $F$, of the Chv\'atal graph}\label{Chvatal-graph-dir2}
\end{center}
\end{figure}

The proof is completed.\end{proof}

\end{document}